\documentclass[a4paper,reqno]{amsart}

\usepackage[utf8]{inputenc}
\usepackage{amssymb}
\usepackage{enumitem}
\usepackage{color}
\usepackage{mathrsfs}
\usepackage{mathtools}
\setlist[enumerate]{label=\emph{(\roman*)}}

\usepackage{hyperref}

\usepackage{environ}
\NewEnviron{rem}{\begin{quote}\small\color{blue}\BODY\end{quote}}

\newtheorem{theorem}{Theorem}[section]

\newtheorem{lemma}[theorem]{Lemma}
\newtheorem{proposition}[theorem]{Proposition}

\theoremstyle{definition}

\newtheorem{remark}[theorem]{Remark}
\numberwithin{equation}{section}
\newcommand{\R}{\mathbb{R}}
\newcommand\inner[2]{\langle #1,#2\rangle}
\newcommand\normt[1]{\left\lVert#1\right\rVert_{L^2}}
\newcommand\normo[1]{\left\lVert#1\right\rVert_{H^1}}
\newcommand\normpro[1]{\left\lVert#1\right\rVert_{E}}
\newcommand{\rd}{{\rm{d}}}
\newcommand\norm[1]{\left\lVert#1\right\rVert}
\newcommand{\pj}{\partial_{x_j}}
\numberwithin{equation}{section}
\begin{document}
	\parindent=0pt
	\title[Two solitary waves with logarithmic distance]{Existence of two-solitary waves with logarithmic distance for the nonlinear Klein-Gordon Equation}
	\author{Shrey Aryan}
\begin{abstract}
We consider the focusing nonlinear Klein-Gordon (NLKG) equation
\begin{equation*}
    \partial_{tt}u - \Delta u + u - |u|^{p-1}u = 0,\quad (t,x)\in \mathbb{R}\times \mathbb{R}^d
\end{equation*}
for $1\leq d\leq  5$ and $p>2$ subcritical for the $\dot H^1$ norm. In this paper we show the existence of a solution $u(t)$ of the equation such that 
\begin{equation*}
    \normo{u(t) - \sum_{k=1,2}Q_k(t)} + \normt{\partial_t u(t)} \to 0\quad
    \mbox{as $t\to +\infty$,}
\end{equation*}
where $Q_k(t,x)$ are two solitary waves of the equation with  translations $z_k:\mathbb{R}\to \mathbb{R}^d$ satisfying 
\begin{equation*}
    |z_1(t) - z_2(t)| \sim 2\log(t)\quad  \text{as } t\to +\infty.
\end{equation*}
This behaviour is due to the strong interactions between solitary waves which is in contrast with the previous work \cite{C_te_2018} on multi-solitary waves of the (NLKG), devoted to the case of solitary waves with different speeds.
The present work is motivated by previous similar existence results for the nonlinear Schr\"odinger and generalized Korteweg-de Vries equations.
\end{abstract}
\maketitle
\section{Introduction}
\subsection{Problem Setup}
We consider the focusing nonlinear Klein-Gordon (NLKG) equation in $\mathbb{R}^d$ for any $1\leq d\leq 5,$
\begin{equation}\label{nlkg}
\partial_{tt}u-\Delta u + u-f(u)=0, \quad (t,x) \in \mathbb{R}\times \mathbb{R}^d
\end{equation} where $f=|u|^{p-1} u$ and $p>2$ is energy subcritical in the sense that,
\begin{equation}\label{on:p}
    2 < p <+\infty \text{ for } d=1,2 \text{ and } 2<p<\frac{d+2}{d-2} \text{ for } d = 3,4,5.
\end{equation}
Rewriting this equation as a first order system with $\vec{u} = (u,\partial_t u ) = (u,v)$ we get,
\begin{equation}\label{nlkg system}
    \begin{cases}
    \partial_t u = v\\
    \partial_t v = \Delta u - u + f(u).
    \end{cases}
\end{equation}
This framework ensures that the corresponding Cauchy problem for \eqref{nlkg} is locally well-posed in $H^1(\mathbb{R}^d)\times  L^2(\mathbb{R}^d)$.
Thus for any initial data $(u_0, v_0) \in H^1(\mathbb{R}^d) \times L^2(\mathbb{R}^d)$ there exists a unique (in some class) maximal solution $u\in C([0, T_{\max}), H^1(\mathbb{R}^d)) \cap C^{1}([0, T_{\max}), L^2(\mathbb{R}^d))$ of $\eqref{nlkg}.$ For proof see the result due to Ginibre and Velo in \cite{local-wellposedness}. Let $F(u) = \frac{1}{p+1}|u|^{p+1}.$ Then recall that the Energy $H$ and Momentum $P$
\begin{align}
    &H[u, \partial_t u](t) = \frac{1}{2}\int \left[|\partial_t u(t,x)|^2 + \left|\nabla u(t,x)\right|^2 + |u(t,x)|^2 - 2F(u(t,x))\right] {\rm{d}}x,\\
    &P[u, \partial_t u](t) = \frac{1}{2} \int \partial_t u(t,x) \nabla u(t,x) {\rm{d}}x,
\end{align}
are conserved along the flow. In this paper, we restrict ourselves to the dynamics of two solitary waves associated with the ground state $Q,$ which is the unique positive radial and $H^1$ solution of the following equation:
\begin{equation}
    \Delta Q-Q+f(Q)=0, \quad Q>0,\quad Q\in H^{1}(\mathbb{R}^d).
\end{equation}
The existence and uniqueness of these solutions is well-known and was proved in \cite{Berestycki1983} and \cite{Q-uniqueness} respectively.  Furthermore, it has also been established that these solutions are radial and exponentially decreasing along with the first and the second derivatives. For the sake of completeness, we recall some properties of $Q$ in Section \ref{ground state section}. 
\subsection{Statement of the Main Result}
We prove the existence of two-solitary wave solutions with logarithmic distance for \eqref{nlkg}. In particular, we have the following theorem. 
\begin{theorem}\label{main theorem}Let $p>2$ satisfying \eqref{on:p}. There exists $T_0>0$ and an $H^1(\mathbb{R}^d) \times L^2(\mathbb{R}^d)$ solution $\vec{u} = (u, \partial_t u)$ of \eqref{nlkg} on $[T_0, +\infty)$ that decomposes asymptotically into two solitary waves,
\begin{equation}\label{main-estimate}
    \norm{u(t) - \sum_{k=1,2} Q(\cdot - z_k(t))}_{H^1} + \normt{\partial_t u(t)} \lesssim t^{-1}
\end{equation}
and 
\begin{equation}\label{log distance}
    |z_1(t) - z_2(t)| = 2(1+ o(1)) \log(t),\quad \text{ as } t\to +\infty.
\end{equation}
\end{theorem}
\noindent
\begin{remark}
Historically, multi-solitary wave solutions (also called multi-solitons) were constructed for integrable equations, like the Korteweg-de Vries equation (see for instance the review paper \cite{kdV multi-soliton}) and 1D cubic Schr\"odinger equation (\cite{1d Schrodinger multi-soliton}). For such integrable system, multi-solitons are special global solutions behaving as the sums of solitons both for $t\to +\infty$ and $t\to -\infty$.
In the nonintegrable setting, such pure multi-solitons are not expected to exist, but one can construct
asymptotic multi-solitary waves, as in the works \cite{merle1990construction} and \cite{martel2006multi}
for the nonlinear Schr\"odinger equation (NLS),  \cite{martel2005asymptotic}
for the generalized Korteweg de Vries equation (gKdV) and 
\cite{cte2012multisolitons} for the nonlinear Klein-Gordon equation.
In those works,  multi-solitary wave solutions exhibit weak interactions for large time.
More precisely, the main result in \cite{cte2012multisolitons} shows that
for any given set of translations $x_1,\cdots, x_K\in \mathbb{R}^d$ and boost velocities $\beta_1,\cdots,\beta_K \in \mathbb{R}^d$, there exists a solution $u(t)$ of \eqref{nlkg} on a semi-infinite interval such that for large enough time
\begin{equation*}
   \norm{(u,\partial_t u)(t) - \sum_{k=1}^{K}(Q_{\beta_k}, \partial_t Q_{\beta_k})(t,\cdot -x_k)}_{H^1\times L^2} \lesssim e^{-\gamma_0 t}
\end{equation*}
where $\gamma_0 >0$ and $Q_{\beta_k}(t,x) = Q\left(\Lambda_{\beta_k}\begin{pmatrix}t\\x\end{pmatrix}\right)$ where we define the matrix $\Lambda_{\beta}$ as follows: for $\beta = (\beta_1, \cdots, \beta_d)^T \in \mathbb{R}^d$ with $|\beta| < 1$ and $\gamma = \frac{1}{\sqrt{1-|\beta|^2}}$ we have,
$$
 \Lambda_{\beta} = 
\begin{pmatrix}   
\gamma & -\beta_1 \gamma & \cdots & \beta_d \gamma \\
-\beta_1 \gamma \\
\vdots   & & \text{Id} + \frac{(\gamma - 1)}{|\beta|^2}\beta \beta^T\\
-\beta_d \gamma 
\end{pmatrix}.$$ 
Later, still inspired by the integrable case (see for instance the discussion in the Introduction of \cite{nguyen2016existence}),
solutions of nonintegrable models containing several solitary waves in strong interaction have been constructed in \cite{Martel-Raphael 2018} for the mass critical (NLS) equation, \cite{nguyen2016existence} for the sub and supercritical (NLS) equation and \cite{nguyen2017strongly} for the (gKdV) equation.
Such strong interaction regimes are quite rigid, the distance between the solitary waves being
of order $\log t$, with apparently no other possible rate. See \cite{jendrej2018dynamics} for an example
of classification result in this framework.
Our motivation was to extend such constructions to the case of the nonlinear Klein-Gordon.
\end{remark}

\begin{remark}
Note that in Theorem \ref{main theorem}, as in the main results of \cite{Martel-Raphael 2018} and \cite{nguyen2016existence}  the solitons have the same sign and are at a $\log$ distance from one another. This behaviour is directly  related to the ODE 
$$\ddot{z}(t) = -2e^{-z(t)}$$
which has $2\log t$ as an exceptional solution with the initial conditions $\left(z(1),\dot{z}(1)\right)= \left(0,2\right)$. Moreover, observe that the energy subcriticality of the exponent $p>2$ allows us to work with finite energy solutions. It is possible to prove a similar result with lower values of $p$ however the analysis is more involved and thus, we do not discuss it here. 
\end{remark}
\subsection{Notation}
We begin by introducing relevant notation. Let $\{\mathbf{e}_1, \cdots, \mathbf{e}_d\}$ be the canonical basis of $\mathbb{R}^d$ and let $\mathcal{B}_{\mathbb{R}^d}(\rho)$ be the closed ball centered at the origin with radius $\rho >0$ for the usual norm $|\xi| = (\sum_{i=1}^{d} \xi_i^2 )^{1/2}$. We define the $L^2$ inner product for any two functions $f,g\in L^2(\mathbb{R}^d)$ and vector valued functions $\vec{f} = (f_1,f_2)$ and $\vec{g} = (g_1,g_2)$ as 
$$\inner{f}{g}_{L^2} = \inner{f}{g} = \int f(x) g(x){\rm{d}}x, \quad \inner{\vec{f}}{\vec{g}} = \int f_1(x) g_1(x){\rm{d}}x + \int f_2(x) g_2(x){\rm{d}}x.$$
Furthermore, for any $f,g\in H^1(\mathbb{R}^d)$ we have the following inner product,
$$\inner{f}{g}_{H^1} = \int \left(f(x)g(x) + \nabla f(x) \cdot \nabla g(x)\right){\rm{d}}x.$$
 Since we will work on the energy space $ E= H^{1}(\mathbb{R}^{d})\times L^2(\mathbb{R}^d)$ we define the inner product on this space for any $\vec{u}=(u_{1},u_{2})$ and $\vec{v}=(v_{1},v_{2})$ with $\vec{u},\vec{v}\in E$ such that, 
$$\inner{\vec{u}}{\vec{v}}_{E} = \inner{u_1}{v_1}_{H^1} + \inner{u_2}{v_2}_{L^2}. $$
\subsection{Ground state and spectral theory}\label{ground state section} The ground state $Q$ can be written as $Q(x) = q(|x|)$ with $q>0$ satisfying the ODE,
\begin{equation}\label{q ODE}
    q'' + \frac{d-1}{r} q' - q + q^{p} = 0, \quad q'(0) = 0, \quad \lim_{r\to +\infty} q(r) = 0.
\end{equation}
Using standard ODE arguments we have the following estimate for some constant $\kappa>0$ and for all $r>1$,
\begin{equation}\label{q bound}
    \left|q(r) - \kappa r^{-\frac{d-1}{2}}e^{-r}\right| + \left|q'(r) + \kappa r^{-\frac{d-1}{2}}e^{-r}\right| \lesssim r^{-\frac{d+1}{2}}e^{-r}.
\end{equation}
See e.g. \cite[\S4.2]{sulem}. Due to the radial symmetry we also have the following properties,
\begin{equation}\label{Q orthogonality}
    \forall i\neq j, \quad \int \partial_{x_i} Q(x) \partial_{x_j}Q(x) {\rm{d}}x = 0, \quad \forall i, j \text{ and } k, \quad \int \partial_{x_i x_j} Q(x) \partial_{x_k}Q(x) {\rm{d}}x = 0.
\end{equation}
Next we recall some spectral properties associated to the linear operator
\begin{equation}
\mathcal{L} = -\Delta + 1 - pQ^{p-1}.
\end{equation}
\begin{lemma}\label{eq:spectral properties}
    \emph{(i) Spectral properties.} The unbounded operator $\mathcal{L}$ on $L^2$ with domain $H^2$ is self-adjoint, its continuous spectrum is $[1, +\infty),$ its kernel is ${\rm{span}}\{\pj Q:j=1,\cdots, d\}$ and it has a unique negative eigenvalue $-\nu_0^2,$ with corresponding smooth normalized radial eigenfunction $Y$ $($$\normt{Y} =1$$)$. Moreover, on $\mathbb{R}^d$
    \begin{equation*}
        |\partial_{x}^{\beta} Y(x)|\lesssim e^{-\sqrt{1+\nu_0^2}|x|} \quad \text{ for any } \beta=(\beta_1,\cdots, \beta_d) \in \mathbb{N}^d.
    \end{equation*}
   \emph{(ii) Coercivity property.} There exists $c>0$ such that for all $\varepsilon\in H^1,$
            \begin{equation}
                \inner{\mathcal{L}\varepsilon}{\varepsilon} \geq c\norm{\varepsilon}_{H^1}^2 - c^{-1}\left(\inner{\varepsilon}{Y}^2 + \sum_{j=1}^{d} \inner{\varepsilon}{\pj Q}^2\right).
            \end{equation}
\end{lemma}
\begin{proof}
See the proof of Lemma $1$ in \cite{cte2012multisolitons}.
\end{proof}
Observe that if we set,
\begin{align*}
    \vec{Z}^{\pm} = \begin{pmatrix}\pm \nu_0 Y \\ Y\end{pmatrix}, \quad \vec{Y}^{\pm} = \begin{pmatrix}
    Y \\
    \pm \nu_0 Y
\end{pmatrix} 
    \end{align*}
then any solution $\vec{\varepsilon}= (\varepsilon,\eta)$ of the system,
\begin{equation}\label{vec_epsilon}
    \begin{cases}
    \partial_t \varepsilon = \eta \\
    \partial_t \eta  = -\mathcal{L}\varepsilon
    \end{cases}
\end{equation}
implies that,
\begin{align*}
    a^{\pm} = \langle \vec{\varepsilon},\vec{Z}^{\pm}\rangle\quad  \text{satisfies}\quad \frac{{\rm{d}} a^{\pm}}
    {{\rm {d}} t} = \pm \nu_0 a^{\pm}
\end{align*}
and the function $\vec{\varepsilon}^{\pm}(t,x) = \exp(\pm \nu_0 t) \vec{Y}^{\pm}(x)$ solves the linearized problem ~\eqref{vec_epsilon}.
\subsection{Outline of the paper.} Following the introduction, in section $2$ we recall some estimates involving the nonlinear interaction term and geometrical parameters associated to the solitary waves. Then in section 3, we prove backward uniform estimates related to the regime of Theorem \ref{main theorem}. The proof relies on a refinement of the energy method introduced in \cite{nguyen2017strongly} and a standard topological argument to control the unstable exponential directions. Finally in section 4 we conclude the proof of Theorem \ref{main theorem} using compactness arguments.

\section{Dynamics of two solitary waves}
Consider time dependent $\mathcal{C}^1$ parameters $(z_1,z_2,\ell_1,\ell_2)$ with $|\ell_1|\ll 1, |\ell_2|\ll 1$ and $|z|\gg 1$ where 
$$z= z_1-z_2 \quad \text{and}\quad \ell = \ell_1-\ell_2.$$
Define the modulated ground state solitary waves for $k=1,2,$ 
$$Q_k = Q(\cdot - z_k)\quad \text{and}\quad \vec{Q}_k = 
\begin{pmatrix}
    Q_k \\
    -(\ell_k\cdot \nabla)Q_k
\end{pmatrix}.$$ Set  
$$R = Q_1+Q_2, \quad  \vec{R} = \vec{Q}_1+\vec{Q}_2$$
and   
\begin{equation*}
G=f(R) - f(Q_1)-f(Q_2),\quad D=-\sum_{k=1,2}(\ell_k\cdot \nabla)^2Q_k.
\end{equation*}
Define the following functions associated to the exponential instabilities around each solitary wave,
for $k=1,2$,
$$Y_k = Y(\cdot -z_k), \quad \vec{Y}^{\pm}_k = \vec{Y}^{\pm}(\cdot -z_k), \quad \vec{Z}_k^{\pm}= \vec{Z}^{\pm}(\cdot -z_k).$$
From Taylor expansion and $p>2$, for any $s\in \R$,
\begin{equation}\label{f taylor expansion f'}
    f(R+ s) = f(R) + f'(R)s + O(|s|^{2}+|s|^p).
\end{equation}

\subsection{Decomposition around 2 solitary waves.}
First, we recall the following technical lemma in~\cite{cte2019description}.
\begin{lemma}\label{nonlinear-interaction-lemma}
The following estimates hold for $|z|\gg1$.
\begin{enumerate}
    \item For any $0<m'<m$,
    \begin{equation}\label{eq:m'}
    \int |Q_1Q_2|^{m}\lesssim e^{-m'|z|},\quad  \int |\nabla Q_1 \nabla Q_2|^{m}\lesssim e^{-m'|z|}
    \end{equation}
    and 
    \begin{equation}\label{eq:F_Estimate}
        \int |F(R)-F(Q_1)-F(Q_2) -f(Q_1)Q_2-f(Q_2)Q_1|\lesssim e^{-\frac{5}{4}|z|}.
    \end{equation}
    \item For any $m>0,$
    \begin{equation}\label{q_1 q_2^1+m estimate}
        \int |Q_1||Q_2|^{1+m}\lesssim q(|z|),
    \end{equation}
    \begin{equation}\label{eq:gbound}
        \normt{G}\lesssim \normt{Q_1^{p-1}Q_2}+ \normt{Q_1Q_2^{p-1}} \lesssim q(|z|). 
    \end{equation}
    \item There exists a smooth function $g:[0,+\infty)\to \mathbb{R}$ such that, for any $0<\theta<\min(p-1,2)$ and $r>1,$
    \begin{equation}\label{eq:sharp_asymp}
        |g(r) - g_0 q(r)|\lesssim r^{-1}q(r)
    \end{equation}
    where 
    \begin{equation}
        g_0 = \frac{1}{c_1} \int Q^p(x) e^{-x_1} {\rm{d}}x > 0,\quad c_1 = \normt{\partial_{x_1} Q}^2
    \end{equation}
    and 
    \begin{equation}\label{eq:proj_q1}
        \left|\inner{G}{\nabla Q_1} - c_1 \frac{z}{|z|}g(|z|)\right|\lesssim e^{-\theta |z|},
    \end{equation}
    \begin{equation}\label{eq:proj_q2}
        \left|\inner{G}{\nabla Q_2} +  c_1 \frac{z}{|z|} g(|z|)\right|\lesssim e^{-\theta|z|}.
    \end{equation}
    \end{enumerate}
\end{lemma} 
\begin{proof}
For parts $\rm{(ii)} , \rm{(iii)}$ see proof of Lemma 2.1 in \cite{cte2019description}.
\newline
Proof of \rm{(i)}. From estimate, ~\eqref{q bound} we have $|Q(y)|\lesssim e^{-|y|}$. Therefore, 
\begin{align*}
    \int\left|Q_{1} Q_{2}\right|^{m} \mathrm{d} y \lesssim \int e^{-m|y|} e^{-m^{\prime}|y+z|} \mathrm{d} y \lesssim e^{-m^{\prime}|z|} \int e^{-\left(m-m^{\prime}\right)|y|} \mathrm{d} y \lesssim e^{-m^{\prime}|z|}.
\end{align*}
Similarly by ~\eqref{q bound} we have $|\nabla Q(y)|\lesssim e^{-|y|}$ and thus by the same argument we get ~\eqref{eq:m'}. For the last part using $p>2$ and Taylor expansion we get,
\begin{align*}
    \left|F\left(Q_{1}+Q_{2}\right)-F\left(Q_{1}\right)-F\left(Q_{2}\right)-f\left(Q_{1}\right) Q_{2}-f\left(Q_{2}\right) Q_{1}\right| \lesssim\left|Q_{1} Q_{2}\right|^{\frac{3}{2}}.
\end{align*}
Therefore, using $\eqref{eq:m'}$ with $m=\frac{3}{2}$ and $m'=\frac{5}{4}$ we get \eqref{eq:F_Estimate}.
\end{proof}
Second, we prove a general decomposition around $2$ solitary waves.
\begin{lemma}[Properties of the decomposition]\label{decomposition around 2 solitons}
There exists $\gamma_0>0$ such that for any $0<\gamma<\gamma_0,$ $T_1\leq T_2$ and any solution $\vec{u} = (u,\partial_t u)$ of \eqref{nlkg} on $[T_1, T_2]$ satisfying 
\begin{equation}\label{decom}
\sup_{t\in [T_1,T_2]}  \bigg(\inf_{|\zeta_1-\zeta_2| > |\log \gamma|}\bigg\|{u- \sum_{k=1,2}Q(\cdot - \zeta_k)}\bigg\|_{H^{1}} + \normt{\partial_t u}\bigg) <\gamma
\end{equation}
there exist unique $\mathcal{C}^1$ functions
$$t\in [T_1, T_2]\to (z_1,z_2,\ell_1,\ell_2)(t)\in \mathbb{R}^{4d}$$
such that the solution $\vec{u}$ decomposes on $[T_1,T_2]$ as
\begin{equation}\label{u vec decomposition}
\vec{u} = \begin{pmatrix}
    u \\
    \partial_t u
\end{pmatrix} = \vec{Q}_1 + \vec{Q}_2 + \vec{\varepsilon}, \quad \vec{\varepsilon} = \begin{pmatrix}
    \varepsilon \\
    \eta
\end{pmatrix}    
\end{equation}
with the following properties on $[T_1, T_2]$.
\begin{enumerate}
    \item \emph{Orthogonality and smallness}. For any $k=1,2$ and $j = 1,\cdots, d$
    \begin{equation}\label{eq:orthogonal}
    \inner{\varepsilon}{\partial_{x_j}Q_k}=\inner{\eta}{\partial_{x_j}Q_k}=0   
    \end{equation}
    and 
    \begin{equation}\label{eq:init}
    \normpro{\vec{\varepsilon}} + \sum_{k=1,2}|\ell_k| + e^{-2|z|}\lesssim \gamma.
    \end{equation}
    \item \emph{Equation of $\vec{\varepsilon}.$}
    \begin{equation}\label{eq:epsilon_t and eta_t}
        \begin{aligned}
        \partial_t \varepsilon&= \eta + {\rm{Mod}_{\varepsilon}}\\
        \partial_t \eta&= \Delta \varepsilon - \varepsilon + f(R+\varepsilon) -f(R) +{\rm{Mod}_{\eta}} + G+D
        \end{aligned}
    \end{equation}
    where
    \begin{equation}\label{mod_eta, mod_epsilon}
    \begin{split}
        \rm{Mod}_{\varepsilon} &= \sum_{k=1,2}(\dot{z}_k - \ell_k)\cdot \nabla Q_k,\\
        \rm{Mod}_{\eta} &= \sum_{k=1,2}\dot{\ell}_k\cdot \nabla Q_k - \sum_{k=1,2} (\ell_k\cdot \nabla )(\dot{z}_k-\ell_{k})\cdot \nabla Q_k .
    \end{split}
    \end{equation}
    \item \emph{Equation of $z_k.$} For $k=1,2$
    \begin{equation}\label{eq:zkdot-lk}
        |\dot{z}_k - \ell_k| \lesssim 
        \normpro{\vec{\varepsilon}} (|\ell_{1}| + |\ell_{2}|).
    \end{equation}
    \item \emph{Equation of $\ell_k.$} For any $1<\theta < \min(p-1,2), m\in (0,1)$ and $k=1,2$
    \begin{equation}\label{eq:refined lkdot}
        \begin{split}
        \left|\dot{\ell}_k -(-1)^k \frac{z}{|z|} g(|z|)\right| &\lesssim  \normpro{\vec{\varepsilon}}^2 +   (|\ell_1| + |\ell_2|)^2 e^{-m|z|} + e^{-\theta |z|}\\ 
        &\quad  + (|\ell_1| + |\ell_2|) \normpro{\vec{\varepsilon}}.   
        \end{split}
    \end{equation}
    \item \emph{Equations of the exponential directions.} Let $a_{k}^{\pm} =\inner{\vec{\varepsilon}}{\vec{Z}_k^{\pm}}$
    then,
    \begin{equation}\label{eq:exponential_directions}
    \left|\frac{{\rm{d}}}{{\rm{d}}t}a_{k}^{\pm} \mp \nu_0 a_{k}^{\pm} \right|\lesssim \normpro{\vec{\varepsilon}}^2 + \sum_{k=1,2} |\ell_k|^2 + q(|z|).
    \end{equation}
    \end{enumerate}
\end{lemma}
\begin{remark}\label{symmetric soliton}
Note that if the solution $u$ of \eqref{nlkg} is symmetric under the map $\tau:x\to -x$ in the context of Lemma \ref{decomposition around 2 solitons} then by uniqueness, the parameters $(z_1,z_2,\ell_1,\ell_2)$ are also symmetric in the sense that $z_1(t) = -z_2(t)$ and $\ell_1(t) = -\ell_2(t)$ for all $t\in [T_1, T_2].$
\end{remark}
\begin{proof} For proof of (i) see proof of Lemma 2.2 in \cite{cte2019description}.
\newline
Proof of (ii). From the definition of $\varepsilon$ and $\eta$ it follows that,
\begin{align*}
\partial_t \varepsilon = \partial_t u - \sum_{k=1,2} \partial_t Q_k = \eta + \sum_{k=1,2} (\dot{z}_k -\ell_k) \cdot \nabla Q_k.
\end{align*}
Next,
\begin{align*}
    \partial_t \eta &= \partial_{tt} u + \sum_{k=1,2} \partial_t (\ell_k \cdot \nabla Q_k)\\
    &= \Delta u -u+f(u)+D + \sum_{k=1,2} \dot{\ell}_k \cdot \nabla Q_k - \sum_{k=1,2} (\ell_k\cdot \nabla) (\dot{z}_k-\ell_{k})\cdot \nabla Q_k.
\end{align*}
Using \eqref{u vec decomposition}, $\Delta Q_k - Q_k + f(Q_k) = 0$ and the definition of $G$,
\begin{align*}
    \Delta u - u + f(u) = \Delta \varepsilon - \varepsilon + f(R+\varepsilon) - f(R) + G.
\end{align*}
Thus we get
\begin{align*}
    \partial_t \eta &= \Delta \varepsilon - \varepsilon + f(R+\varepsilon) - f(R) +  G+D\\
    &\quad+\sum_{k=1,2} \dot{\ell}_k \cdot (\nabla Q_k)-\sum_{k=1,2} (\ell_k\cdot \nabla) (\dot{z}_k-\ell_{k})\cdot \nabla Q_k.
\end{align*}
\newline
Proof of (iii). From~\eqref{eq:orthogonal} and \eqref{eq:epsilon_t and eta_t} for $j=1,\cdots, d,$
\begin{align*}
\frac{\rd}{\rd t}\inner{\varepsilon}{\partial_{x_j}Q_1} &= \inner{\partial_t \varepsilon}{\partial_{x_j} Q_1} + \inner{\varepsilon}{\partial_t (\partial_{x_j} Q_1)}\\
&= \inner{\eta}{\partial_{x_j} Q_1} + \inner{{\rm{Mod}}_{\varepsilon}}{\partial_{x_j} Q_1} - \inner{\varepsilon}{\dot{z}_1\cdot \nabla \partial_{x_j} Q_1}\\
&= \inner{{\rm{Mod}}_{\varepsilon}}{\partial_{x_j} Q_1} - \inner{\varepsilon}{\dot{z}_1\cdot \nabla \partial_{x_j} Q_1} = 0.
\end{align*}
Therefore,
\begin{align*}
    |\dot{z}_{1,j}-\ell_{1,j}|\normt{\partial_{x_j}Q}^2 &\lesssim  |\dot{z}_2 - \ell_2|\int |\nabla Q_2(x)| |\nabla Q_1(x)| {\rm{d}}x  + |\dot{z}_1| \normt{\varepsilon}
\end{align*}
and so using \eqref{eq:m'} with $m=1$ and $m'=1/2$,
\begin{align*}
|\dot{z}_1-\ell_1| &\lesssim |\dot{z}_2 - \ell_2| e^{-\frac{1}{2}|z|} + |\dot{z}_1 - \ell_1| \normpro{\vec{\varepsilon}} + |\ell_1|\normpro{\vec{\varepsilon}}.
\end{align*}
Since $\normpro{\vec{\varepsilon}}\lesssim \gamma$ we have,
\begin{align*}
|\dot{z}_1-\ell_1| &\lesssim |\dot{z}_2 - \ell_2| e^{-\frac{1}{2}|z|} + |\ell_1|\normpro{\vec{\varepsilon}}.
\end{align*}
Similarly, it holds
\begin{align*}
    |\dot{z}_2 - \ell_2| \lesssim |\dot{z}_1 - \ell_1|e^{-\frac{1}{2}|z|} + |\ell_2| \normpro{\vec{\varepsilon}}.
\end{align*}
Thus for large $|z|$,
\begin{align*}
    |\dot{z}_k -\ell_k|\leq \sum_{k=1,2} |\dot{z}_k - \ell_k| &\lesssim (|\ell_1|+|\ell_2|)\normpro{\vec{\varepsilon}}.
\end{align*}
Proof of (iv). Using again \eqref{eq:orthogonal} and \eqref{eq:epsilon_t and eta_t} for $j=1, \cdots, d,$
 \begin{align*}
    &\frac{\rd}{\rd t}\inner{\eta}{\partial_{x_j} Q_1} = \inner{\partial_t\eta}{\partial_{x_j} Q_1} + \inner{\eta}{\partial_t(\partial_{x_j} Q_1)}\\
    &=\inner{\Delta \varepsilon - \varepsilon + f'(Q_1) \varepsilon}{\partial_{x_j} Q_1} + \inner{f(R+\varepsilon)-f(R)-f'(R)\varepsilon}{\partial_{x_j} Q_1} \\ 
    &+ \inner{(f'(R)-f'(Q_1))\varepsilon}{\partial_{x_j} Q_1}+ \inner{{\rm{Mod}}_{\eta}}{\partial_{x_j} Q_1} + \inner{G}{\pj Q_1} + \inner{D}{\pj Q_1} \\
    &- \inner{\eta}{\left(\dot{z}_1\cdot \nabla\right)\partial_{x_j} Q_1}.
\end{align*}
Observe that since $\pj Q_1$ satisfies $\Delta \pj Q_1 - \pj Q_1 + f'(Q_1)\pj Q_1  = 0,$ the first term is zero. For the second term, we use \eqref{f taylor expansion f'} and the $H^1$
sub-criticality of the exponent $p > 2$ to get,
\begin{align}\label{f(R+E)-f(R)-f'(R)E}
    |\inner{f(R+\varepsilon)-f(R) -f'(R)\varepsilon}{\pj Q_1}| \lesssim \normo{\varepsilon}^2.
\end{align}
For the third term, first note that using $p>2$ and Taylor expansion we get,
\begin{align*}
  \sum_{k=1,2}|f'(R) - f'(Q_k)||\partial_{x_j} Q_k| \lesssim |Q_2||Q_1|^{p-1} + |Q_1||Q_2|^{p-1}.
\end{align*}
Thus using \eqref{eq:gbound},
\begin{align}\label{f'(R)-f'(Q_1)}
    \left|\inner{(f'(R)-f'(Q_1))\varepsilon}{\pj Q_1}\right|
    \lesssim q(|z|)\normpro{\vec{\varepsilon}}.
\end{align}
For the fourth term, first observe that
\begin{align*}
    \inner{{\rm{Mod}}_\eta}{\pj Q_1} &= \dot{\ell}_{1,j}\normt{\pj Q_1}^2 + \inner{\dot{\ell}_2\cdot \nabla Q_2}{\pj Q_1} \\
    &\quad - \sum_{k=1,2} \inner{(\ell_k\cdot \nabla)((\dot{z}_k-\ell_k)\cdot \nabla) Q_k}{\pj Q_1}.
\end{align*}
Thus using Lemma \ref{nonlinear-interaction-lemma} and ~\eqref{Q orthogonality} for any $m\in (0,1),$
\begin{align*}
    \inner{{\rm{Mod}_{\eta}}}{\pj Q_1} &= \dot{\ell}_{1,j}\normt{\pj Q_1}^2  + O\left(\left|\dot{\ell}_2\right| e^{-m|z|}+ |\ell_2|\left|\dot{z}_2-\ell_2\right|e^{-m|z|}\right).
\end{align*}
Also using \eqref{eq:zkdot-lk} we get,
\begin{align*}
    |\ell_2|\left|\dot{z}_2-\ell_2\right| \lesssim (|\ell_1| + |\ell_2|) \normpro{\vec{\varepsilon}}
\end{align*}
and therefore,
\begin{align*}
    \inner{{\rm{Mod}_{\eta}}}{\pj Q_1} 
    &= \dot{\ell}_{1,j}\normt{\pj Q_1}^2  + O\left(|\dot{\ell}_2|e^{-m|z|}
    +(|\ell_1| + |\ell_2|) \normpro{\vec{\varepsilon}}\right).
\end{align*}
Using \eqref{eq:proj_q1} for $1<\theta<\min(p-1,2)$ we have
\begin{align*}
    \left|\frac{\inner{G}{\pj Q_1}}{\normt{\partial_{x_1}Q}^2} -   \frac{z_j}{|z|}g(|z|)\right|\lesssim e^{-\theta |z|}.
\end{align*}
Next using the definition of $D$ we get,
\begin{align*}
    \left|\inner{D}{\pj Q_1}\right| \lesssim |\ell_2|^2 e^{-m|z|} \lesssim (|\ell_1| + |\ell_2|)^2e^{-m|z|}.
\end{align*}
For the last term, we use \eqref{eq:zkdot-lk},
\begin{align*}
    \left| \inner{\eta}{(\dot{z}_1\cdot \nabla) \pj Q_1}\right|
     \lesssim (|\ell_1| + |\ell_2|)\normpro{\vec{\varepsilon}}.
\end{align*}
Gathering all these estimates,
\begin{align*}
\left|\dot{\ell}_1 +   \frac{z}{|z|} g(|z|)\right| &\lesssim |\dot{\ell}_2| e^{-m|z|} + \normpro{\vec{\varepsilon}}^2  + (|\ell_1| + |\ell_2|)^2e^{-m|z|} \\
&\quad + (|\ell_1| + |\ell_2|)\normpro{\vec{\varepsilon}} +e^{-\theta |z|}.
\end{align*}
Similarly by using $\inner{\eta}{\pj Q_2} = 0,$
\begin{align*}
\left|\dot{\ell}_2 -   \frac{z}{|z|} g(|z|)\right| &\lesssim |\dot{\ell}_1| e^{-m|z|} + \normpro{\vec{\varepsilon}}^2  + (|\ell_1| + |\ell_2|)^2e^{-m|z|} \\
&\quad + (|\ell_1| + |\ell_2|)\normpro{\vec{\varepsilon}} +e^{-\theta |z|}.
\end{align*}
For large enough $|z|$, these estimates imply \eqref{eq:refined lkdot}.
\newline
Proof of (v). By definition of $a_{1}^{\pm}$ we have,
\begin{align*}
    \frac{\rd}{\rd t}a_1^{\pm} &= \inner{\partial_t \vec{\varepsilon}}{\vec{Z}_1^{\pm}} + \inner{\vec{\varepsilon}}{\partial_t \vec{Z}_{1}^{\pm}}\\
    &= \pm \nu_0\inner{\partial_t \varepsilon}{ Y_1} + \inner{\partial_t \eta}{Y_1} + \inner{\vec{\varepsilon}}{\partial_t \vec{Z}_{1}^{\pm}}\\
    &= \pm \nu_0\inner{\eta}{ Y_1} \pm \nu_0\inner{{\rm{Mod}_{\varepsilon}}}{ Y_1} + \inner{\Delta \varepsilon - \varepsilon + f'(Q_1)\varepsilon}{Y_1}\\
    &\quad  + \inner{f(R+\varepsilon)-f(R)-f'(R)\varepsilon}{Y_1} + \inner{(f'(R)-f'(Q_1))\varepsilon}{Y_1}\\
    &\quad + \inner{{\rm{Mod}_{\eta}}}{Y_1} + \inner{G}{Y_1} + \inner{D}{Y_1} -\inner{\vec{\varepsilon}}{(\dot{z}_1\cdot \nabla)\vec{Z}_1^{\pm}}.
\end{align*}
Since $\mathcal{L}Y = -\nu_{0}^{2} Y,$ 
\begin{align*}
\pm \nu_0 \inner{\eta}{Y_1} + \inner{\Delta \varepsilon - \varepsilon + f'(Q_1)\varepsilon}{Y_1} = \pm \nu_0 a_{1}^{\pm}.
\end{align*}
Furthermore using the decay properties of $Y$ in Lemma \ref{eq:spectral properties} along with \eqref{f(R+E)-f(R)-f'(R)E}, \eqref{f'(R)-f'(Q_1)} and the Cauchy-Schwarz inequality we get,
\begin{align*}
\left|\inner{f(R+\varepsilon)-f(R)-f'(R)\varepsilon}{Y_1}\right| + \left|\inner{(f'(R)-f'(Q_1))\varepsilon}{Y_1}\right|  \lesssim  \normt{\varepsilon}^2 + e^{-\frac{3}{2}|z|}.
\end{align*}
From \eqref{eq:gbound} and definition of $D$ we get,
\begin{align*}
    |\inner{G}{Y_1}| +|\inner{D}{Y_1}| \lesssim q(|z|) + \sum_{k=1,2}|\ell_k|^2.
\end{align*}
For the remaining terms using \eqref{mod_eta, mod_epsilon}, \eqref{eq:zkdot-lk} and \eqref{eq:refined lkdot},
\begin{align*} 
\left|\inner{\rm{Mod}_{\varepsilon}}{Y_1}\right|  + \left|\inner{\rm{Mod}_{\eta}}{Y_1}\right| + |\inner{\vec{\varepsilon}}{(\dot{z}_1\cdot \nabla) \vec{Z}_{1}^{\pm}}| \lesssim \normpro{\vec{\varepsilon}}^2 + \sum_{k=1,2}|\ell_k|^2 + q(|z|).
\end{align*}
Thus combining the above estimates we get \eqref{eq:exponential_directions} for $k=1.$ We can similarly prove the estimate for $k=2.$ 
\end{proof}
\section{Backward Uniform Estimates}
Using the estimates derived in the previous section we now proceed to prove Theorem \ref{main theorem}. We argue by compactness, which allows us to show that $u(t)$ (a solution of \eqref{nlkg}) is  asymptotically equal to the sum of $2$ solitons. We first recall the following lemma in order to setup the initial conditions for the bootstrap argument.
\begin{lemma}\label{W lemma}
Let $\beta = -\frac{1}{2\nu_0} = \inner{\vec{Y}^-}{\vec{Z}^-}^{-1}<0.$ For any $(z_1,z_2,\ell_1,\ell_2)\in \mathbb{R}^{4d}$ with $|z|$ large enough, there exist linear maps,
$$B:\mathbb{R}^2 \to \mathbb{R}^2, \quad V_j:\mathbb{R}^2\to \mathbb{R}^2\quad \text{for } j=1,\cdots, d,$$
smooth in $(z_1,z_2,\ell_1,\ell_2)$ satisfying 
$$\norm{B-\beta {\rm{Id}}} \lesssim e^{-\frac{1}{2}|z|}, \quad \norm{V_j}\lesssim e^{-\frac{1}{2}|z|},$$
and such that the function $W(a_1,a_2):\mathbb{R}^{d}\to \mathbb{R}$ defined by 
$$W(a_1,a_2) = \sum_{k=1,2}\left[ B_k(a_1,a_2) Y_k + \sum_{j=1}^{d} V_{k,j}(a_1,a_2) \pj Q_k\right]$$
satisfies for all $k=1,2$ and $j=1,\cdots, d,$
$$\inner{W(a_1,a_2)}{\pj Q_k} = 0, \quad \inner{W(a_1,a_2)}{Y_k} = \beta a_k.$$
In particular setting, $\vec{W}(a_1,a_2) = \begin{pmatrix} W(a_1,a_2) \\ 
                                      -\nu_0 W(a_1,a_2)
                        \end{pmatrix}$ it holds that 
                        
\begin{align*}
    \inner{\vec{W}(a_1,a_2)}{\vec{Z}^{-}_k} = a_k \quad \text{and}\quad \inner{\vec{W}(a_1,a_2)}{\vec{Z}^{+}_k} = 0.
\end{align*}
\end{lemma}

\begin{proof} The existence of the linear maps follows from inverting a linear system for $|z|$ large enough. For proof see Lemma 4.1 in \cite{cte2019description}. For the last part using the definition of $\beta,$
\begin{align*}
    \inner{\vec{W}(a_1,a_2)}{\vec{Z}^{-}_k} = \inner{W(a_1,a_2)}{-\nu_0 Y_k} + \inner{-\nu_0 W(a_1,a_2)}{  Y_k} = -2\nu_0 \beta a_k = a_k
\end{align*}
and 
\begin{align*}
    \inner{\vec{W}(a_1,a_2)}{\vec{Z}^{+}_k} = \inner{W(a_1,a_2)}{\nu_0 Y_k} + \inner{-\nu_0 W(a_1,a_2)}{  Y_k} = 0.
\end{align*}
\end{proof}
Let $(T_n)_{n\geq 1}$ be an increasing sequence of $\mathbb{R}^{+}$ with $\lim_{n\to +\infty} T_{n} = +\infty$. Let $\bar{z}_{n}\in \R$ and $(a_{k,n})_{k=1,2} \in \mathbb{R}^2$ to be determined later. 
For any large $n$, we consider the solution $\vec{u}_{n}$ of~\eqref{nlkg} with initial data
\begin{equation}\label{defini}
\vec{u}_{n}(T_{n})=\vec{Q}_{1,n}+\vec{Q}_{2,n}+\vec{W}(a_{1,n},a_{2,n}),
\end{equation}
where
\begin{equation*}
\vec{Q}_{k,n}= \begin{pmatrix}
    Q(\cdot -z_{k,n})\\
    -(\ell_{k, n}\cdot \nabla)Q(\cdot -z_{k,n})
\end{pmatrix}
\end{equation*}
for $k=1,2.$ In lieu of Remark \ref{symmetric soliton}, we assume that $z_{1,n}(t) =-z_{2,n}(t)=\frac{1}{2}z_{n}(t).$ We claim the following uniform backward estimates.
\begin{proposition}\label{uniform backward estimates}
There exist $n_{0}>0$ and $T_0>0$ large enough, such that for all $n\ge n_{0}$, there exist $\bar{z}_n > 0$,  $a_{k,n} \in \mathcal{B}_{\mathbb{R}^2}(T_n^{-3/2})$ for $k=1,2$
with 
\begin{equation}\label{intial data}
\begin{split}
&\left|(\kappa g_0)^{-\frac{1}{2}}(\bar{z}_n)^{\frac{d-1}{4}} e^{\frac{1}{2}\bar{z}_n} - T_n\right| < T_n\log^{-\frac{1}{2}}T_n, \quad |z_n(T_n)| = \bar{z}_n,
\\
&\ell_{k,n}(T_n) =  (-1)^{k+1}\sqrt{\kappa g_0}(\bar{z}_n)^{-\frac{d-1}{4}}e^{-\frac{1}{2}\bar{z}_n} \mathbf{e}_1, \quad \vec{\varepsilon}_n(T_n) = \vec{W}(a_{1,n}^{-}(T_n), a_{2,n}^{-}(T_n)) \\
&a_{k,n}^{-}(T_n) = a_{k,n}, \quad a_{k,n}^{+}(T_n) =0
\end{split}
\end{equation}
and initial decomposition
\begin{align*}
    \vec{u}_{n}(T_n) = \vec{Q}_{1,n}(T_n) + \vec{Q}_{2,n}(T_n) + \vec{W}(a_{1,n}^{-}(T_n), a_{2,n}^{-}(T_n)) 
\end{align*}
such that the corresponding solution $\vec{u}_n$ of \eqref{nlkg} exists on $[T_0,T_n]$,
satisfies the decomposition of Lemma \ref{decomposition around 2 solitons},
\begin{equation*}
    \vec{u}_n(t,x) = \sum_{k=1,2}\vec{Q}_n(x-z_{k,n}(t)) + \vec{\varepsilon}_n(t,x)
\end{equation*}
and verifies the following uniform estimates for all $t\in [T_0,T_n]$
\begin{equation}\label{uniform bounds}
\begin{split}
    &||z_n(t)| - 2\log(t)| \lesssim \log\log t, \quad |\ell_n (t)|\lesssim t^{-1},\quad \normpro{\vec{\varepsilon}_n(t)} \lesssim t^{-1}\log^{-3/2}t.
\end{split}
\end{equation}
\end{proposition}
\noindent
For the sake of simplicity we drop the index $n$ (except for $T_n$) in the following sections. Our goal now is to prove Proposition \ref{uniform backward estimates} using a bootstrap argument, integration of a differential system of geometrical parameters and energy estimates.
\subsection{Bootstrap setting}
The proof of Proposition \ref{uniform backward estimates} is based on the following bootstrap estimates, for $C^{*}\gg1$ to be chosen later,
\begin{equation}\label{bootstrap}
\begin{split}
|(\kappa g_0)^{-1/2}|z(t)|^{\frac{d-1}{4}}e^{\frac{|z(t)|}{2}} - t| \leq t \log^{-1/2}t, \quad \sum_{k=1,2}|\ell_k(t)|\leq 4t^{-1}, \\
 \sum_{k=1,2} |a_k^{+}(t)| \leq t^{-3/2},\quad \sum_{k=1,2} |a_k^{-}(t)| \leq t^{-3/2},\quad \normpro{\vec{\varepsilon}(t)} \leq C^{*}t^{-1}\log^{-3/2}t.
\end{split}
\end{equation}
Note that the estimate on $z$ gives us a more precise estimate,
\begin{equation}\label{z implies inequalities}
  |z(t)| = 2\log t - \frac{d-1}{2}\log\log t - C + O(\log^{-1/2}t),
\end{equation}
where $C>0$ is a constant depending only on $d$ and $p.$ We can thus deduce, 
\begin{align*}
   ||z(t)| - 2\log t|\lesssim \log\log t.
\end{align*}
Let,
$$T^{*}(\bar{z}_{n},a_{1,n},a_{2,n})= \inf\{t\in [T_0, T_n]: \vec{u}(t)\ \mbox{satisfies}~\eqref{decom}\ \mbox{and}~\eqref{bootstrap}\ \mbox{on}\ [t,T_n]\},$$
where $\vec{u}$ is the solution of~\eqref{nlkg} with initial data $\vec{u}_{n}(T_{n})$ given by~\eqref{defini}. 
Next we derive some inequalities that will allow us to improve the bootstrap estimates.
\subsection{Modulation Equations.}
\begin{lemma}\label{control of modulation equations}
	For all $t\in [T^{*},T_{n}]$, the following hold.
	\begin{enumerate}
		\item Estimates on $z_{k}$ and $z$. We have
	\begin{equation}\label{est:z}
	\sum_{k=1,2}\left|\dot{z}_k - \ell_k\right| \lesssim t^{-2}\log^{-1}t,\quad 
	|q(|z|)|\lesssim t^{-2}.
	\end{equation}
	\item Estimates on $\ell_{k}$. We have
	\begin{equation}\label{est:l}
\sum_{k=1,2}\left||\dot{\ell}_k| - t^{-2}\right|\lesssim t^{-2}\log^{-1}t,\quad 
	\sum_{k=1,2}\left||\ell_k| - t^{-1}\right|\lesssim t^{-1}\log^{-1}t.
	\end{equation}
	\item Estimates on $\partial_{t}G$ and $\partial_{t}D$. We have 
	\begin{equation}\label{est:ptGD}
	\|\partial_{t}G\|_{L^{2}}+\|\partial_{t}D\|_{L^{2}}\lesssim t^{-3}.
	\end{equation}
\end{enumerate}
\end{lemma}

\begin{proof}
	Proof of~\eqref{est:z}. For the first estimate, using \eqref{eq:zkdot-lk} and \eqref{bootstrap} for large enough $T_0$ depending on $C^*$,
\begin{align*}
       \sum_{k=1,2} |\dot{z}_k - \ell_k| &\lesssim (|\ell_1|+|\ell_2|)\normpro{\vec{\varepsilon}}\lesssim C^* t^{-2}\log^{-3/2}t \lesssim t^{-2}\log^{-1}t.
\end{align*} For the second estimate, observe that from \eqref{bootstrap} and \eqref{q bound} we have,
\begin{align*}
        \left|q(|z|)\right| &\lesssim  \left|q(|z|) - \kappa |z|^{-\frac{d-1}{2}}e^{-|z|}\right| + \kappa |z|^{-\frac{d-1}{2}}e^{-|z|} \lesssim t^{-2}.
\end{align*}

Proof of~\eqref{est:l}. For the first estimate note that using \eqref{eq:sharp_asymp}, \eqref{eq:refined lkdot}, \eqref{bootstrap} and \eqref{z implies inequalities} we get, 
\begin{align*}
    \left| \dot{\ell}_k - g_0(-1)^{k} \frac{z}{|z|}q(|z|)\right| &\lesssim \left| \dot{\ell}_k - (-1)^{k} \frac{z}{|z|}g(|z|)\right| + \left|g(|z|) - g_0q(|z|)\right|\\
    &\lesssim \normpro{\vec{\varepsilon}}^2  +   (|\ell_1| + |\ell_2|)^2 e^{-m|z|} + e^{-\theta |z|}  \\
    &\quad  + (|\ell_1| + |\ell_2|) \normpro{\vec{\varepsilon}} + |z|^{-1}q(|z|)\\
    &\lesssim \normpro{\vec{\varepsilon}}^2 + t^{-2-2m} + t^{-2\theta} + t^{-1}\normpro{\vec{\varepsilon}}\\
    &\quad  + t^{-2}\log^{-1}t\lesssim t^{-2} \log^{-1}t,
\end{align*}
where $1<\theta<\min(p-1,2), $ $m\in (0,1)$ and $k=1,2.$ Thus using \eqref{q bound},
\begin{align*}
    &\left|\dot{\ell}_k - \kappa g_0(-1)^k \frac{z}{|z|}|z|^{-\frac{d-1}{2}}e^{-|z|}\right|\\
    \lesssim& \left|\dot{\ell}_k -g_0(-1)^k\frac{z}{|z|}q(|z|)\right| + \left|q(|z|)-\kappa |z|^{-\frac{d-1}{2}}e^{-|z|}\right| \lesssim t^{-2}\log^{-1}t.
\end{align*}
On using \eqref{bootstrap} we get,
$\sum_{k=1,2} \left| |\dot{\ell}_k| - t^{-2}\right|\lesssim t^{-2} \log^{-1}t.$ Note that using the triangle inequality we can also infer that,
\begin{equation}\label{ldot +2g_0}
    \left|\dot{\ell} + 2\kappa g_0 \frac{z}{|z|}|z|^{-\frac{d-1}{2}}e^{-|z|}\right|
    \lesssim t^{-2}\log^{-1}t,
\end{equation}
which in turn implies that $\left| |\dot{\ell}| - 2t^{-2}\right|\lesssim t^{-2} \log^{-1}t.$ We will use these estimates on $\ell$ in later sections.

For the second estimate, by integrating $\left| |\dot{\ell}_k| - t^{-2}\right|\lesssim t^{-2} \log^{-1}t$ on the interval $[t, T_n]$ with the initial data in \eqref{intial data} we get, $\sum_{k=1,2} \left| |\ell_k| - t^{-1}\right|\lesssim t^{-1} \log^{-1}t.$ 

Proof of~\eqref{est:ptGD}. By direct computation,
$$\partial_t{D} = -\sum_{j=1,2} \left[2(\dot{\ell}_j\cdot \nabla )(\ell_j\cdot \nabla Q_j)- (\ell_j\cdot \nabla )(\ell_j\cdot \nabla )(\dot{z}_j\cdot \nabla Q_j)\right].$$
Thus using \eqref{est:l} and \eqref{est:z} we get $\normt{\partial_t D}\lesssim t^{-3}.$ Similarly,
$$\partial_t{G} = -f'(Q_1+Q_2)(\dot{z}_1 \cdot \nabla Q_1 + \dot{z}_2 \cdot \nabla Q_2) + \sum_{k=1,2} f'(Q_k)(\dot{z}_k\cdot \nabla Q_k).$$
Thus using \eqref{f'(R)-f'(Q_1)} and \eqref{est:z} we get 
$$\normt{\partial_t G}\lesssim q(|z|)\sum_{k=1,2}|\dot{z}_k|\lesssim t^{-3}$$ and hence we obtain the desired inequality.

\end{proof}
\subsection{Energy estimates}
Consider the energy functional for $\vec{\varepsilon}=(\varepsilon,\eta)$,
\begin{equation}
        \mathcal{E}(t) = \int \{ |\nabla \varepsilon|^2 + \varepsilon^2 + \eta^2   -2[F(R+\varepsilon) -F(R)-f(R)\varepsilon]\}.
    \end{equation}
Let $\chi:[0,+\infty)\to [0,+\infty)$ be a smooth and non-increasing function with $\chi \equiv 1$ on $[0,1/10],$ $\chi \equiv 0$ on $[1/8, +\infty).$ Set
\begin{equation*}
    \mathcal{J}(t) = \sum_{k=1,2} J_k(t),\quad J_k(t) =  \int (\ell_k\cdot\nabla \varepsilon)\eta \chi_k,
\end{equation*}
where
\begin{equation*}
\chi_{k}(t,x)=\chi\left(\log^{-1}(t)|x-z_k(t)|\right),\quad \mbox{for}\ k=1,2.
\end{equation*}
Note that, from~\eqref{est:z},
\begin{equation}\label{est:derchi}
|\partial_{t}\chi_{k}(t,x)|\lesssim \frac{\boldsymbol{1}_{\Omega_{k}}(t,x)}{t\log t},\quad 
|\nabla_{x} \chi_{k}(t,x)|\lesssim\frac{\boldsymbol{1}_{\Omega_{k}}(t,x)}{\log t},
\end{equation}
where
\begin{equation*}
\Omega_{k}(t,x)=\left\{x\in \R^{d}:|x-z_{k}(t)|\le \frac{1}{8}\log t\right\}.
\end{equation*}
Let 
\begin{equation}\label{defn: S}
    \mathcal{S}(t) = \inner{G(t)}{\varepsilon(t)}+\inner{D(t)}{\varepsilon(t)}.
\end{equation}
Last, we set
\begin{equation*}
    \mathcal{W}(t) = \mathcal{E}(t) + 2 \mathcal{J}(t) - 2\mathcal{S}(t).
\end{equation*}
\begin{lemma}
For all $t \in [T^*, T_{n}]$, the following hold.
    \begin{enumerate}
        \item \emph{Coercivity.}
        \begin{equation}\label{coercivity}
             \normpro{\vec{\varepsilon}(t)}^2 \lesssim \mathcal{W}(t) +O(t^{-3}).
        \end{equation}
        \item \emph{Time variation.}
        \begin{equation}\label{time variation}
           \left|\frac{\rd}{\rd t}\mathcal{W}\right| \lesssim C^{*}t^{-3}\log^{-3}t.
        \end{equation}
    \end{enumerate}
Note that in both estimates, the implicit constant is independent of $C^{*}.$
\end{lemma}
\begin{proof}
Proof of (i).
From~\eqref{bootstrap} and~\eqref{est:l}, for $T_{0}$ large enough,
\begin{align*}
    |\mathcal{J}(t)| \lesssim \sum_{k=1,2}|\ell_{k}|\normpro{\vec{\varepsilon}}^2 \lesssim (C^{*})^{2}t^{-3}\log ^{-3}t \lesssim t^{-3}.
\end{align*}
Using~\eqref{eq:gbound} and~\eqref{bootstrap} for large enough $T_0,$
\begin{equation*}
|\mathcal{S}(t)|\lesssim q(|z|)\normpro{\vec{\varepsilon}} + \sum_{k=1,2}|\ell_k|^2 \normpro{\vec{\varepsilon}} \lesssim C^{*}t^{-3}\log^{-3/2}t\lesssim t^{-3}.
\end{equation*}
Therefore the coercivity property for $\mathcal{W}$ is a consequence of the following 
coercivity property for $\mathcal{E}$, there exists a constant $\nu>0$ such that,
\begin{equation}\label{coer:E}
\nu\normpro{\vec{\varepsilon}}^2 \leq \mathcal{E}(t) + \frac{t^{-3}}{\nu}. 
\end{equation}
The proof of coercivity \eqref{coer:E} follows from the coercivity property in Lemma \ref{eq:spectral properties} around one solitary wave along with \eqref{eq:orthogonal} and a localization argument. We refer to Lemma 2.4 in \cite{cte2019description} for the proof of the following estimate using \eqref{bootstrap}. Thus we can deduce \eqref{coercivity} for large enough $T_0$ depending on $C^*$.
\newline
%%%%%% Estimate for energy and momentum %%%%%%%%%%

Proof of (ii). \textbf{Step 1}. Time variation of $\mathcal{E}$. We claim
\begin{equation}\label{energy time control}
\begin{split}
\frac{\rd}{\rd t}\mathcal{E} &=2\sum_{k=1,2} \langle \ell_{k}\cdot \nabla Q_{k},f(R+\varepsilon)-f(R)-f'(R)\varepsilon\rangle\\
&\quad +2\inner{\eta}{G} +2\inner{\eta}{D}
+O((C^{*})^2t^{-4}\log ^{-3}t).
\end{split}
\end{equation}
In order to see this first observe that,
\begin{align*}
\frac{{\rm{d}}}{{\rm{d}}t}\mathcal{E} &=2\int \partial_t \varepsilon\left[ -\Delta \varepsilon + \varepsilon -f(R+\varepsilon) +f(R)\right] + \eta \partial_t \eta \\
&\quad + 2\sum_{k=1,2} \int (\dot{z}_k \cdot \nabla Q_k)[f(R+\varepsilon)-f(R)-f'(R)\varepsilon]=\mathcal{I}_1 + \mathcal{I}_2.
\end{align*}
\emph{Estimate on $\mathcal{I}_{1}$.} We claim
\begin{equation}\label{est:I1}
\mathcal{I}_{1}=2\inner{\eta}{G}+2 \inner{\eta}{D}
+O\left((C^{*})^2t^{-4}\log^{-3}t\right).
\end{equation}
Using \eqref{eq:epsilon_t and eta_t} and then \eqref{eq:orthogonal}
\begin{align*}
    \mathcal{I}_1 &=2\inner{\eta}{G} + 2 \inner{\eta}{D}+\mathcal{I}_{1,1}+\mathcal{I}_{1,2},
\end{align*}
where
\begin{equation*}
\mathcal{I}_{1,1}=-2\sum_{k=1,2} \inner{\eta}{(\ell_k\cdot \nabla )((\dot{z}_k-\ell_k)\cdot \nabla) Q_k},
\end{equation*}
\begin{equation*}
\mathcal{I}_{1,2}=2\int {\rm{Mod}}_{\varepsilon}\left(-\Delta \varepsilon + \varepsilon -f(R+\varepsilon)+f(R)\right).
\end{equation*}
From~\eqref{eq:zkdot-lk},~\eqref{bootstrap} and Cauchy-Schwarz inequality,
\begin{equation*}
\left|\mathcal{I}_{1,1}\right|\lesssim \sum_{k=1,2}|\ell_{k}|^{2}\|\vec{\varepsilon}\|_{E}^{2}\lesssim
(C^{*})^{2}t^{-4}\log ^{-3}t.
\end{equation*}
Using integration by parts and $\Delta \partial_{x_j}Q_k - \partial_{x_j}Q_k + f'(Q_k)\partial_{x_j}Q_k=0$ for $k=1,2$ and $j=1,\cdots,d$,
\begin{equation*}
\begin{aligned}
&\langle \partial_{x_j}Q_{k},-\Delta \varepsilon+\varepsilon-f(R+\varepsilon)+f(R)\rangle\\
=&-\langle \partial_{x_j} Q_{k},f(R+\varepsilon)-f(R)-f'(R)\varepsilon\rangle
-\langle \partial_{x_j}Q_{k},\left(f'(R)-f'(Q_{k})\right)\varepsilon\rangle.
\end{aligned}
\end{equation*}
Note that, for $j=1,\cdots,d$,
$$\sum_{k=1,2}|f'(R) - f'(Q_k)||\partial_{x_j} Q_k| \lesssim |Q_2||Q_1|^{p-1} + |Q_1||Q_2|^{p-1}.$$
Therefore using ~\eqref{eq:zkdot-lk}, ~\eqref{f(R+E)-f(R)-f'(R)E} and~\eqref{f'(R)-f'(Q_1)}, 
\begin{equation*}
\left|\mathcal{I}_{1,2}\right|\lesssim \sum_{k=1,2}\left|\dot{z}_{k}-\ell_{k}\right|\|\vec{\varepsilon}\|_{E}\left(\|\vec{\varepsilon}\|^2_{E}
+\|(f'(R)-f'(Q_{k}))\nabla Q_{k}\|_{L^{2}}\right)\lesssim (C^{*})^2 t^{-4}\log ^{-3}t
\end{equation*}
for $T_{0}$ large enough. Combining the above estimates we obtain \eqref{est:I1}.

\emph{Estimate on $\mathcal{I}_{2}$.} From ~\eqref{eq:zkdot-lk}, ~\eqref{f(R+E)-f(R)-f'(R)E} and~\eqref{bootstrap},
\begin{equation}\label{est:I2}
\begin{aligned}
&\bigg|\mathcal{I}_{2}-2\sum_{k=1,2} \langle \ell_{k}\cdot \nabla Q_{k},f(R+\varepsilon)-f(R)-f'(R)\varepsilon\rangle\bigg|\lesssim \sum_{k=1,2}|\dot{z}_{k}-\ell_{k}| \normo{\varepsilon}^2  \\
&\lesssim (C^{*})^2 t^{-4}\log ^{-3}t
\end{aligned}
\end{equation}
for large enough $T_0.$ Gathering~\eqref{est:I1} and~\eqref{est:I2}, we obtain~\eqref{energy time control}.

\textbf{Step 2.} Time variation of $\mathcal{J}$. We claim
\begin{equation}\label{momentum inequality}
    \begin{split}
    \frac{{\rm{d}}}{{\rm{d}}t}\mathcal{J} &=- \sum_{k=1,2}\inner{\ell_k\cdot \nabla Q_k }{f(R+\varepsilon)-f(R)-f'(R)\varepsilon}+O\big(C^{*}t^{-3}\log^{-3}t\big).
    \end{split}
\end{equation}
By direct computation and integration by parts, for $k=1,2$,
\begin{equation*}
\frac{\rm{d}}{{\rm{d}} t} J_{k}=\mathcal{I}_{k,1}+\mathcal{I}_{k,2}+\mathcal{I}_{k,3},
\end{equation*}
where
\begin{equation*}
\begin{split}
\mathcal{I}_{k,1} &=\inner{\dot{\ell}_k \cdot \nabla \varepsilon}{\eta \chi_k}+
\inner{{\ell}_k \cdot \nabla \varepsilon}{\eta (\partial_{t}\chi_k)}, \quad  \mathcal{I}_{k,2}=\inner{\ell_{k}\cdot\nabla(\partial_{t}\varepsilon)}{\eta\chi_{k}}, \\
\mathcal{I}_{k,3}&=
\inner{{\ell}_k \cdot \nabla \varepsilon}{\chi_{k}(\partial_{t}\eta)}.
\end{split}
\end{equation*}
\emph{Estimate on $\mathcal{I}_{k,1}$.} We claim
\begin{equation}\label{est:Ik1}
\left|\mathcal{I}_{k,1}\right|\lesssim (C^{*})^{2}t^{-4}\log^{-3}t.
\end{equation}
By~\eqref{bootstrap} and~\eqref{est:l}, we have 
\begin{align*}
    \left| \inner{\dot{\ell}_k \cdot \nabla \varepsilon}{\eta \chi_k}\right| \lesssim t^{-2} \normpro{\vec{\varepsilon}}^2\lesssim (C^{*})^{2}t^{-4}\log^{-3}t.
\end{align*}
From~\eqref{bootstrap},~\eqref{est:l} and~\eqref{est:derchi},
\begin{align*}
    \left|\ell_k \cdot  \inner {\nabla \varepsilon}{\eta (\partial_t  \chi_k)}  \right| \lesssim |\ell_{k}|\|\partial_{t}\chi_{k}\|_{L^{\infty}}\normpro{\vec{\varepsilon}}^2\lesssim
    (C^{*})^{2}t^{-4}\log^{-4}t.
\end{align*}
Gathering the above estimates, we obtain~\eqref{est:Ik1}.

\emph{Estimate on $\mathcal{I}_{k,2}$.}
By~\eqref{eq:epsilon_t and eta_t} and integration by parts,
\begin{equation*}
\mathcal{I}_{k,2}=-\frac{1}{2}\inner{\eta^{2}}{(\ell_{k}\cdot\nabla) \chi_{k}}
+\inner{(\ell_{k}\cdot\nabla){\rm{Mod}}_{\varepsilon}}{\eta\chi_{k}}.
\end{equation*}
Thus, from~\eqref{eq:zkdot-lk},~\eqref{bootstrap} and~\eqref{est:derchi},
\begin{equation}\label{est:Ik2}
\left|\mathcal{I}_{k,2}\right|
\lesssim |\ell_{k}|\|\nabla \chi_{k}\|_{L^{\infty}}\|\vec{\varepsilon}\|_{E}^{2}
+\sum_{k=1,2}|\ell_{k}|^{2}\|\vec{\varepsilon}\|_{E}^{2}\lesssim (C^{*})^{2}t^{-3}\log ^{-4}t.
\end{equation}

\emph{Estimate on $\mathcal{I}_{k,3}$.}
We claim
\begin{equation}\label{est:Ik3}
\mathcal{I}_{k,3}=-\inner{\ell_{k}\cdot\nabla Q_{k}}{f(R+\varepsilon)-f(R)-f'(R)\varepsilon}
+O(C^{*}t^{-3}\log^{-3}t).
\end{equation}
Now we use \eqref{eq:epsilon_t and eta_t} to get,
\begin{align*}
\mathcal{I}_{k,3}&= \inner{\Delta \varepsilon-\varepsilon}{\left(\ell_{k}\cdot\nabla \varepsilon\right) \chi_k} + \inner{G}{\left(\ell_{k}\cdot\nabla \varepsilon\right) \chi_k} + \inner{D}{\left(\ell_{k}\cdot\nabla \varepsilon\right) \chi_k}\\
& \quad +\inner{f(R+\varepsilon)-f(R)}{\left(\ell_{k}\cdot\nabla \varepsilon\right) \chi_k}+\inner{{\rm{Mod}}_{\eta}}{(\ell_{k}\cdot\nabla\varepsilon)\chi_{k}}
\end{align*}
For the first term, using integration by parts, for $j=1,\cdots,d$,
\begin{align*}
    \inner{\Delta \varepsilon-\varepsilon}{\chi_k\partial_{x_j}\varepsilon}
    &=-\frac{1}{2}\inner{(\partial_{x_j} \varepsilon)^{2}-\varepsilon^{2} }{\partial_{x_{j}} \chi_k}\\
    &\quad -\sum_{j'\ne j}\langle(\partial_{x_{j'}}\varepsilon)(\partial_{x_{j}}\varepsilon),\partial_{x_{j'}}\chi_{k} \rangle
    +\frac{1}{2}\sum_{j'\ne j}\inner{(\partial_{x_{j'}}\varepsilon)^{2}}{\partial_{x_{j}}\chi_{k}}.
\end{align*}
Thus from~\eqref{est:l}, ~\eqref{est:derchi} and ~\eqref{bootstrap},
\begin{equation*}
\left|\inner{\Delta \varepsilon-\varepsilon}{\left(\ell_{k}\cdot\nabla \varepsilon\right) \chi_k}\right|
\lesssim |\ell_{k}|  \|\nabla \chi_{k}\|_{L^{\infty}}  \|\vec{\varepsilon}\|_{E}^{2}\lesssim (C^{*})^{2}t^{-3}\log^{-4}t.
\end{equation*}
For the second term, using \eqref{eq:gbound}, \eqref{bootstrap}, \eqref{est:z} and~\eqref{est:l},
\begin{align*}
\left| \ell_k \cdot \inner{G}{\left(\nabla \varepsilon\right) \chi_k} \right| &\lesssim t^{-1} q(|z|) \normpro{\vec{\varepsilon}} \lesssim C^{*}t^{-4}\log^{-\frac{3}{2}}t.
\end{align*}
For the third term, using \eqref{bootstrap} and \eqref{est:l}
\begin{align*}
\left| \ell_k \cdot \inner{D}{\left(\nabla \varepsilon\right) \chi_k} \right| &\lesssim t^{-3}  \normpro{\vec{\varepsilon}} \lesssim C^{*}t^{-4}\log^{-\frac{3}{2}}t.
\end{align*}
For the fourth term we use integration by parts to get,
\begin{align*}
    \inner{f(R+\varepsilon)-f(R)}{( \ell_k\cdot\nabla \varepsilon)\chi_k} &= - \inner{F(R+\varepsilon)-F(R)-f(R)\varepsilon}{\ell_k\cdot\nabla \chi_k} \\
    &\quad -\inner{f(R+\varepsilon)-f(R)-f'(R)\varepsilon}{(\ell_k\cdot\nabla R)\chi_k}.
\end{align*}
Next by Taylor expansion, ~\eqref{bootstrap},~\eqref{est:l}, ~\eqref{est:derchi} and the $H^1$ sub-criticality of the exponent $p>2$ we get,
\begin{align*}
    \left|\ \inner{F(R+\varepsilon)-F(R)-f(R)\varepsilon}{\ell_k\cdot\nabla \chi_k}\right|\lesssim (C^{*})^{2}t^{-3}\log^{-4}t.
\end{align*}
Using the decay properties of $Q$ and the definition of $\chi_{k}$ we obtain,
\begin{equation*}
\begin{aligned}
&-\inner{f(R+\varepsilon)-f(R)-f'(R)\varepsilon}{(\ell_k\cdot\nabla R)\chi_k}\\
&=-\inner{\ell_{k}\cdot\nabla Q_{k}}{f(R+\varepsilon)-f(R)-f'(R)\varepsilon}
+O((C^{*})^{2}t^{-3}\log^{-4}t).
\end{aligned}
\end{equation*}
For the last term, from~\eqref{eq:zkdot-lk},~\eqref{est:l} and~\eqref{bootstrap}, we have
\begin{equation*}
\left|\inner{{\rm{Mod}}_{\eta}}{(\ell_{k}\cdot\nabla \varepsilon)\chi_{k}}\right|
\lesssim \sum_{k'=1,2}\left(|\ell_{k'}|+\|\vec{\varepsilon}\|_{E}\right)(|\ell_{k'}|\|\vec{\varepsilon}\|_{E}+|\dot{\ell}_{k'}|)\lesssim (C^{*})^{2}t^{-3}\log^{-4}t.
\end{equation*}
Gathering the above estimates for large enough $T_0$, we obtain~\eqref{est:Ik3}. We see that ~\eqref{momentum inequality} follows from ~\eqref{est:Ik1}, ~\eqref{est:Ik2} and ~\eqref{est:Ik3}.

\textbf{Step 3.} Time variation of $\mathcal{S}$. We claim the following estimate. 
\begin{equation}\label{correction term estimate}
    \frac{{\rm{d}}}{{\rm{d}}t}\mathcal{S} = \inner{G}{\eta} +\inner{D}{\eta}+ O(C^{*}t^{-4}\log^{-\frac{3}{2}}t). 
\end{equation}
From~\eqref{eq:epsilon_t and eta_t} and the definition of $\mathcal{S}$ in \eqref{defn: S},
\begin{align*}
\frac{{\rm{d}}}{{\rm{d}}t}\mathcal{S} &= \inner{\partial_t {G}}{\varepsilon} + \inner{G}{\eta} +   \inner{G}{{\rm{Mod}}_{\varepsilon}}  + \inner{\partial_t{D}}{\varepsilon} +\inner{D}{\eta} + \inner{D}{{\rm{Mod}}_{\varepsilon}} \\
& = \mathcal{S}_1 + \inner{G}{\eta} + \mathcal{S}_2 +\mathcal{S}_{3}+\inner{D}{\eta} + \mathcal{S}_{4}.
\end{align*}
\emph{Estimate on $\mathcal{S}_{1}$.} Using \eqref{est:ptGD} we get,
\begin{align*}
    |\mathcal{S}_1| = \left|\inner{\partial_t {G}}{\varepsilon}\right|  \lesssim t^{-3} \normpro{\vec{\varepsilon}}\lesssim C^{*}t^{-4}\log^{-\frac{3}{2}}t.
\end{align*}
\emph{Estimate on $\mathcal{S}_{2}$.} Using \eqref{est:z}, \eqref{mod_eta, mod_epsilon}, \eqref{eq:zkdot-lk} and \eqref{est:l} we get,
\begin{align*}
    |\mathcal{S}_2| = \left|\inner{G}{{\rm{Mod}}_{\varepsilon}}\right| \lesssim t^{-2}(|\ell_1|+|\ell_2|)\normpro{\vec{\varepsilon}} \lesssim t^{-3}\normpro{\vec{\varepsilon}}\lesssim C^{*}t^{-4}\log^{-\frac{3}{2}}t.
\end{align*}
\emph{Estimate on $\mathcal{S}_{3}$}. Using \eqref{est:ptGD} we get,
$$|\mathcal{S}_{3}| = |\inner{\partial_t{D}}{\varepsilon}| \lesssim t^{-3} \normpro{\vec{\varepsilon}}\lesssim C^{*}t^{-4}\log^{-\frac{3}{2}}t.$$
\emph{Estimate on $\mathcal{S}_{4}$}. Finally from the definition of $D$ and \eqref{est:l} we get $|D|\lesssim t^{-2}.$ Thus using, \eqref{mod_eta, mod_epsilon}, \eqref{eq:zkdot-lk} and \eqref{est:l} we get,
$$|\mathcal{S}_{4}| = \left|\inner{D}{{\rm{Mod}}_{\varepsilon}}\right|\lesssim t^{-3}\normpro{\vec{\varepsilon}} \lesssim C^{*}t^{-4}\log^{-\frac{3}{2}}t.$$ 
Combining these estimates gives us \eqref{correction term estimate}.
\newline
\textbf{Step 4.} Conclusion. From \eqref{energy time control}, \eqref{momentum inequality} and \eqref{correction term estimate} we get \eqref{time variation}.
\end{proof}
\subsection{Closing Bootstrap Estimates}
Using the estimates on the energy functional we can now improve all the estimates in \eqref{bootstrap} except the ones on $z$ and $(a_k^{-})_{k=1,2}.$
\begin{lemma} 
For $C^*>0$ large enough, for all $t\in [T^*, T_n]$ we have,
\begin{equation}\label{improved bootstrap}
    \normpro{\vec{\varepsilon}} \leq \frac{C^{*}}{2}t^{-1}\log^{-3/2}t,\quad 
        \sum_{k=1,2}|\ell_k| \leq 3 t^{-1}\quad \text{and} \quad \sum_{k=1,2} |a_k^{+}| \leq \frac{1}{2}t^{-3/2}.
\end{equation}
\end{lemma}
\begin{proof}
\textit{Estimate on} $\vec{\varepsilon}.$ For this we use the energy functional. We first note that by integrating \eqref{time variation} on the interval $[t, T_n]$ and using the initial data in \eqref{intial data} we get,
\begin{align*}
    \left|\frac{{\rm{d}}}{{\rm{d}} t}\mathcal{W}\right| \lesssim C^{*}t^{-3}\log^{-3}(t) 
    \implies |\mathcal{W}| \lesssim C^{*} t^{-2}\log^{-3}t.
 \end{align*}
Thus by the coercivity of $\mathcal{W}$ in \eqref{coercivity} we have,
\begin{align*}
\normpro{\vec{\varepsilon}}^2  \leq C_0 \left(C_1 C^{*} t^{-2}\log^{-3}t+C_2t^{-3}\right),
\end{align*}
where $C_0,C_1$ and $C_2$ are constants not depending on $C^{*}.$ Therefore for large enough $T_0$ depending on $C_1$ and $C_2$ such that $C_2 T_0^{-0.5}\leq C_1$ we get for $C^{*}\gg 1,$
\begin{align*}
\normpro{\vec{\varepsilon}}^2  \leq C_0C_1 \left( C^{*} t^{-2}\log^{-3}t+t^{-2.5}\right) \leq 2C_3 C^{*}t^{-2}\log^{-3}t
\end{align*}
where $C_3 = C_0C_1.$ Then for large enough $C^{*}$ such that $C_3C^{*}\leq \frac{(C^{*})^{2}}{8}$, the estimate on $\vec{\varepsilon}$ can be strictly improved.
\newline
\textit{Estimate on} $\ell_k.$ Using the estimate on $\ell_k$ in \eqref{est:l} we get,
\begin{align*}
    |\ell_k| \leq t^{-1}(1 + C_0\log^{-1}(t))
\end{align*}
and so for large enough $T_0$ depending on $C_0$ such that $C_0\log^{-1}(T_0)\leq \frac{1}{2}$ we get $|\ell_k|\leq  \frac{3}{2} t^{-1}$ which implies that $\sum_{k=1,2} |\ell_k| \leq 3 t^{-1}.$ 
\newline
\textit{Estimate on} $a_k^+.$
Using \eqref{eq:exponential_directions} and the initial data in \eqref{intial data} we get, 
\begin{align*}
    |a^{+}_k(t)| &\lesssim (C^{*})^2 e^{\nu_0 t}\int_{t}^{T_n} e^{-\nu_0 \tau } \tau^{-2} d\tau\leq \frac{(C^*)^2}{\nu_0} t^{-2} \leq \frac{1}{4}t^{-3/2}
\end{align*}
for $T_0$ to be large enough.
\end{proof}
\noindent
Next, we use a topological argument to close the bootstrap on the parameters $z$ and $a_k^{-}$ for $k=1,2.$ Following the strategy in \cite{cote-martel-merle}, \cite{combet2017construction} and \cite{nguyen2016existence} we will show that there exists a choice of initial data, $(a_k)_{k=1,2} \in \mathcal{B}_{\mathbb{R}^2}( T_n^{-3/2})$ and $\bar{z}>0$ such that $T^{*} = T_0.$ This argument will thus conclude the proof of Lemma \ref{uniform backward estimates}.

\begin{lemma}
There exists $n_0\geq 1$ large enough, such that for $n\geq n_0$ there exists $\bar{z}_n>0$ and $a_{k,n} \in \mathcal{B}_{\mathbb{R}^2}( T_n^{-3/2})$ for $k=1,2$ such that $T^{*} = T_0.$
\end{lemma}
\begin{proof}
We prove this claim by contradiction. Let
\begin{align*}
 \zeta(t) = (\kappa g_0)^{-1/2} |z(t)|^{\frac{d-1}{4}} e^{\frac{1}{2}|z(t)|}, \xi(t) = (\zeta(t) - t)^2 t^{-2} \log t
\end{align*}
and for unstable direction $a^{-}(t) = (a^{-}_1(t), a^{-}_2(t))$ denote,
\begin{align*}
     &\mathcal{N}(t) = \sum_{k=1,2} t^{3} (a_k^{-}(t))^2 = t^3 |a^{-}(t)|^2.
\end{align*}
Suppose that for all $(\hat{\zeta}, \hat{a}) = (\hat{\zeta}, \hat{a}_1, \hat{a}_2) \in \mathbb{D} = [-1,1] \times \mathcal{B}_{\mathbb{R}^2}(1)$ the choice,
\begin{equation*}
    \zeta(T_n) = T_n + \hat{\zeta} T_n \log^{-1/2}T_n,\quad a^{-}(T_n) = \hat{a}(T_n)^{-3/2}
\end{equation*}
gives us $T^{*} = T^{*}(\hat{\zeta}, \hat{a}) \in (T_0, T_n].$ Now observe that,
\begin{align}\label{xi derivative}
\dot{\xi}(t) = 2(\zeta(t) - t)(\dot{\zeta}(t) - 1)t^{-2} \log t - (\zeta(t)-t)^2 (2t^{-3} \log t - t^{-3}).
\end{align}
In order to estimate $\dot{\xi}$ we first claim that for all $t\in (T^{*}(\hat{\zeta}, \hat{a}), T_n]$,
\begin{equation}\label{z derivative}
    \left|\frac{\rd}{\rd t}\left(|z|^{\frac{d-1}{4}} e^{\frac{1}{2}|z|}\right) - \sqrt{\kappa g_0}\right|\lesssim \log^{-1}t.
\end{equation}
Using \eqref{ldot +2g_0} we get,
\begin{align*}
    \left|\dot{\ell} \cdot \frac{z}{|z|} + 2\kappa g_0 |z|^{-\frac{d-1}{2}}e^{-|z|}\right|\lesssim t^{-2}\log^{-1}t.
\end{align*}
Similarly using \eqref{est:z} along with triangle inequality we get $\left|\dot{z}-\ell\right|\lesssim t^{-2}\log^{-1}t$ which in turn implies that
\begin{align*}
    \left|\dot{z}\cdot \frac{z}{|z|}-\ell\cdot \frac{z}{|z|}\right|\lesssim t^{-2}\log^{-1}t.
\end{align*}
Therefore using \eqref{est:l}
\begin{align*}
    \left|\left(\dot{\ell}\cdot \frac{z}{|z|} \right)\left(\ell\cdot \frac{z}{|z|}\right) + 2\kappa g_0 \dot{z}\cdot \frac{z}{|z|}|z|^{-\frac{d-1}{2}}e^{-|z|}\right|\lesssim t^{-3}\log^{-1}t.
\end{align*}
Thus using the initial data in \eqref{intial data} we can integrate on the interval $[t, T_n]$ where $t\in [T^{*}(\hat{\zeta}, \hat{a}),T_n]$ to get (when $d-1>0$)
\begin{align*}
    \left|\frac{1}{2}\left(\ell\cdot \frac{z}{|z|}\right)^2 - 2\kappa g_0 |z|^{-\frac{d-1}{2}}e^{-|z|} \right|\lesssim t^{-2}\log^{-1}t
\end{align*}
and when $d-1=0,$
\begin{align*}
     \left|\left(\dot{\ell}\cdot \frac{z}{|z|} \right)\left(\ell\cdot \frac{z}{|z|}\right) + 2\kappa g_0 \dot{z}\cdot \frac{z}{|z|}e^{-|z|}\right| &\lesssim t^{-3}\log^{-1}t\\
     \implies \left|\frac{1}{2}\left(\ell\cdot \frac{z}{|z|}\right)^2 - 2\kappa g_0 e^{-|z|}\right| &\lesssim t^{-2}\log^{-1}t.
\end{align*}
Thus we have,
\begin{align*} 
    \left|\left(\ell\cdot \frac{z}{|z|}\right) - \sqrt{4\kappa g_0} |z|^{-\frac{d-1}{4}}e^{-\frac{1}{2}|z|} \right| + \left| \left(\dot{z}\cdot \frac{z}{|z|}\right)-\left(\ell\cdot \frac{z}{|z|}\right)\right| \lesssim t^{-1}\log^{-1}t
\end{align*}
which implies,
\begin{align*}
    \left|\left(\dot{z}\cdot \frac{z}{|z|}\right) - \sqrt{4\kappa g_0} |z|^{-\frac{d-1}{4}}e^{-\frac{1}{2}|z|} \right| \lesssim t^{-1}\log^{-1}t.
\end{align*}
Next for $d-1>0,$
\begin{align*}
\frac{\rd}{\rd t}\left(|z|^{\frac{d-1}{4}} e^{\frac{1}{2}|z|}\right) = \frac{1}{2}\dot{z}\cdot \frac{z}{|z|} |z|^{\frac{d-1}{4}} e^{\frac{1}{2}|z|} + \frac{d-1}{4}\dot{z}\cdot \frac{z}{|z|} |z|^{\frac{d-1}{4}-1}e^{\frac{1}{2}|z|}
\end{align*}
and $d-1=0,$
\begin{align*}
    \frac{\rd}{\rd t}(e^{\frac{1}{2}|z|}) = \frac{1}{2}\dot{z}\cdot \frac{z}{|z|} e^{\frac{1}{2}|z|}.
\end{align*}
Thus,
\begin{align*}
    \left|\frac{\rd}{\rd t}\left(|z|^{\frac{d-1}{4}} e^{\frac{1}{2}|z|}\right) - \sqrt{\kappa g_0}\right|\lesssim \log^{-1}(t)+ \frac{d-1}{4} |\dot{z}||z|^{\frac{d-1}{4}-1} e^{\frac{1}{2}|z|}\lesssim  \log^{-1}t.
\end{align*}
In other words we have,
\begin{equation}\label{zeta derivative}
    \left|\dot{\zeta}(t) - 1 \right| \lesssim \log^{-1}t.
\end{equation}
The above inequality will allow us to estimate $\dot{\xi}.$ On the other hand, to estimate $\dot{\mathcal{N}}$ we use \eqref{bootstrap} and \eqref{eq:exponential_directions}, to get that for all $t\in (T^{*}(\hat{\zeta}, \hat{a}), T_n]$
\begin{equation}
\begin{split}
\dot{\mathcal{N}}(t) &= \sum_{k=1,2} t^{3} \left(3t^{-1} a_k^{-}(t) + 2\frac{da_k^{-} }{dt}(t) \right)a_k^{-}(t) \\
&= \sum_{k=1,2} t^{3}  \left(3t^{-1}   - 2\nu_0\right) a_k^{-}(t)^2 +  O\left(\normpro{\vec{\varepsilon}}^2 t^{3}|a_k^{-}(t)| + t|a_k^{-}(t)|\right) \\
&\leq \left(3t^{-1}   - 2\nu_0\right) \mathcal{N}(t) + C t^{-1/2} ((C^*)^2 \log^{-3}(t) + 1)\sqrt{\mathcal{N}(t)},
\end{split}
\end{equation}
where $C>0$ is a constant. Thus for $T_0$ large enough depending on $C$ and $C^{*},$
\begin{equation}\label{N dot estimate}
    \dot{\mathcal{N}}(t) \leq -\frac{3\nu_0}{2}\mathcal{N}(t) + \frac{\nu_0}{2}\sqrt{\mathcal{N}(t)}.
\end{equation}
Let,
\begin{equation*}
    \Psi_1(t) = (\zeta(t) - t) t^{-1} \log^{1/2}t,\quad 
    \Psi_{2}(t) = (a^{-}_1(t) t^{3/2}, a^{-}_2(t)t^{3/2})  = a^{-}(t)(t)^{3/2}.
\end{equation*}
From the definition of $T^{*}$ and continuity of the flow, at the limit $T^{*}(\hat{\zeta}, \hat{a})$ we have one of the following situation
\begin{align*}
\Psi_1(T^{*}(\hat{\zeta}, \hat{a})) &= \pm 1, \quad \Psi_{2}  \in \mathcal{B}_{\mathbb{R}^2}(1)\\
|\Psi_{2}(T^{*}(\hat{\zeta}, \hat{a}))| &= 1\iff\Psi_{2} \in \partial \mathcal{B}_{\mathbb{R}^2}(1) , \quad \Psi_1 \in [-1,1].
\end{align*}
where $\partial \mathcal{B}_{\mathbb{R}^2}(1)$ is the boundary of the closed ball $\mathcal{B}_{\mathbb{R}^2}(1).$ In the first case using \eqref{xi derivative} and \eqref{zeta derivative} we get,
\begin{align*}
    \left|\dot{\xi}\left(T^{*}(\hat{\zeta}, \hat{a})\right)+2(T^{*}(\hat{\zeta}, \hat{a}))^{-1}\right| \lesssim \left(T^{*}(\hat{\zeta}, \hat{a})\right)^{-1}\log^{-1/2}\left(T^{*}(\hat{\zeta}, \hat{a})\right)
\end{align*}
and so
\begin{align*}
    \dot{\xi}(T^{*}(\hat{\zeta}, \hat{a})) < - T^{*}(\hat{\zeta}, \hat{a})^{-1} < 0
\end{align*}
for large enough $T_0$ depending on $C^*$ whereas in the second case we have that $\mathcal{N}(T^{*}(\hat{\zeta}, \hat{a}))=1$ and so from \eqref{N dot estimate} we get, 
\begin{align*}
\mathcal{\dot{N}}_k(T^{*}(\hat{\zeta}, \hat{a})) \leq -\frac{1}{2}\nu_0 <0.
\end{align*}
This transversality property implies the continuity of the map $(\hat{\zeta}, \hat{a})\to T^{*}((\hat{\zeta}, \hat{a}))$ and hence the following map,
\begin{align*}
    \Psi: \mathbb{D} &\to \partial \mathbb{D}\\
        (\hat{\zeta}, \hat{a}) &\to (\Psi_1(T^{*}(\hat{\zeta}, \hat{a})),  \Psi_{2}(T^{*}(\hat{\zeta}, \hat{a}))
\end{align*}
is also continuous where $\partial \mathbb{D}$ denotes the boundary of $\mathbb{D}.$ Note that if $\hat{a} \in \partial \mathcal{B}_{\mathbb{R}^2}(1)$ then \eqref{N dot estimate} implies that $\dot{\mathcal{N}}(T_n)<0,$ we have $T^{*}(\hat{\zeta}, \hat{a}) = T_n$ and if $\hat{\zeta} = \pm 1$ then 
from \eqref{zeta derivative} we get $\dot{\xi}(T_n)<0$ and $T^{*}(\hat{\zeta}, \hat{a}) = T_n.$ Thus $\Psi(\hat{\zeta}, \hat{a}) = (\hat{\zeta}, \hat{a})$ for all $(\hat{\zeta}, \hat{a}) \in \partial \mathbb{D}.$ Therefore the restriction of $\Psi$ to the boundary of $\mathbb{D}$ is the identity. However this contradicts the no retraction theorem. Thus there exists initial data $(\bar{z}, a_k)$ for $k=1,2$ such that $T^{*} = T_0.$ 
\end{proof}
\section{Compactness Argument}
\subsection{Construction of a sequence of backwards solutions.}
Recall the following lemma:
\begin{lemma}\label{weak limit of solution}
The {\rm{(NLKG)}} flow is continuous for the weak $H^1\times L^2$ topology. More precisely, let $\vec{u}_n\in \mathcal{C}([0, T], H^1\times L^2)$ be a sequence of solutions to the {\rm{(NLKG)}} and assume that for some $M>0,$
\begin{align*}
    \vec{u}_n(0) \rightharpoonup \vec{u}^{*}\text{ in }H^1\times L^2\text{-- weak, and}\quad \forall n,\quad \norm{\vec{u}_n(t)}_{\mathcal{C}([0, T], H^1\times L^2)} \leq M. 
\end{align*}
Define $\vec{u}\in \mathcal{C}([0, T^{+}(\vec{u})), H^1\times L^2)$ be the solution to the {\rm{(NLKG)}} with initial data $\vec{u}(0) = \vec{u}^{*}.$ Then $T^{+}(\vec{u}) > T$ and
\begin{align*}
    \forall t\in [0,T], \quad \vec{u}_n(t)\rightharpoonup \vec{u}(t) \text{ in } H^1\times L^2 \text{- weak}.
\end{align*}
\end{lemma}
\begin{proof}
See proof of Lemma 10 in  \cite{cte2012multisolitons}.
\end{proof}
We can now finish the proof of Theorem \ref{main theorem}. From Proposition \ref{uniform backward estimates} there exists a sequence of final data functions $\vec{u}_{0,n}\in H^{1}\times L^2$ such that,
\begin{equation*}
    \forall t\in [T_0,T_n], \quad \vec{u}_n(t) = \phi(T_n, t, \vec{u}_{0,n}) 
\end{equation*}
where $\phi = (u, \partial_t u)^T$ is the flow of \eqref{nlkg}. Note that $T_0$ does not depend on $T_n$ and that there exists $M>0$ independent of $n$ such that 
$$\forall t\in [T_0, T_n],\quad \normpro{\vec{u}_n(t) - \vec{R}_n(t)}\leq M t^{-1}\log^{-3/2}t.$$
Let $\vec{u}^*_0$ be a weak limit in $H^1\times L^2$ of the bounded sequence $\vec{u}_n(T_0)$ up to a subsequence extraction and define,
\begin{equation*}
    \vec{u}^{*}(t) = \phi(t, T_0, \vec{u}^{*}_0).
\end{equation*}
Fix $t\geq T_0.$ Then using Lemma \ref{weak limit of solution} on $[T_0, t]$ we get $T^{+}(\vec{u}^*)>t$ and $\vec{u}_n(t)\rightharpoonup \vec{u}^{*}(t)$ weakly in $H^1\times L^2.$ Moreover, note that the estimates \eqref{est:z} and \eqref{est:l} provide uniform bounds on the time derivatives of the geometric parameters $\left({z}_{k,n}(t), \ell_{k, n}(t)\right)_{k=1,2}$ on the interval $[T_0, T_n].$ Therefore by Ascoli Lemma we get uniform convergence as $n\to +\infty,$ 
\begin{align*}
   \left({z}_{k,n}(t), \ell_{k, n}(t)\right)_{k=1,2}\to  ({{z}}_{k}(t), {\ell}_{k}(t))_{k=1,2}
\end{align*}
on compact subsets of $[T_0,+\infty)$ up to a subsequence extraction for some continuous functions $({{z}}_{k}(t), {\ell}_{k}(t))_{k=1,2}.$  Thus as $n\to +\infty$,
\begin{align*}
\vec{\varepsilon}_n(t) \rightharpoonup \vec{\varepsilon}(t)\text{ in } H^1\times L^2\text{- weak}
\end{align*}
for $t\in [T_0,+\infty).$ Therefore for $t\in [T_0,+\infty)$, using the uniform estimates due to Proposition \ref{uniform backward estimates} we get,
\begin{equation}
    ||z(t)| - 2\log t|\lesssim \log \log t, \quad |\ell(t)|\lesssim t^{-1},\quad \normpro{\vec{\varepsilon}(t)} \lesssim t^{-1}\log^{-3/2}t,
\end{equation}
which imply that $|z(t)| = 2(1+o(1))\log(t)$ as $t\to +\infty.$ Finally, using the relation $\partial_t u = \eta - \sum_{k=1,2}(\ell_k\cdot \nabla )Q_k$,
\begin{align*}
    \normo{{u}_n(t) - \sum_{k=1,2}{Q}\left(\cdot -z_{k,n}(t)\right)} + \normt{\partial_t u_n(t)} \lesssim \normpro{\vec{\varepsilon}_n(t)} + \sum_{k=1,2}|\ell_{k,n}(t)|  \lesssim t^{-1}
\end{align*}
and passing to the limit as $n\to +\infty$ we get \eqref{main-estimate} which concludes the proof of Theorem \ref{main theorem}.
 
\subsection*{Acknowledgements} This work is finished under the guidance of Yvan~Martel and Xu~Yuan to whom the author owes great gratitude.

\end{document}